\documentclass[12pt,twoside]{amsart}

\usepackage{amsmath,amsthm,amscd,amssymb,mathrsfs,graphicx,amsfonts,mathrsfs}
\usepackage{amsfonts}
\usepackage{amssymb,enumerate}
\usepackage{amsthm}
\usepackage[all]{xy}
\usepackage{hyperref}
\allowdisplaybreaks[4] \footskip=12pt
\renewcommand{\uppercasenonmath}[1]{}

\numberwithin{equation}{section} \theoremstyle{plain}
\newtheorem*{thm*}{Main Theorem}
\newtheorem{thm}{Theorem}[section]
\newtheorem{cor}[thm]{Corollary}
\newtheorem*{cor*}{Corollary}
\newtheorem{lem}[thm]{Lemma}
\newtheorem*{lem*}{Lemma}

\newtheorem*{fact*}{Fact}

\newtheorem*{nota*}{Notation}
\newtheorem{prop}[thm]{Proposition}
\newtheorem*{prop*}{Proposition}
\newtheorem{rem}[thm]{Remark}
\newtheorem*{rem*}{Remark}

\newtheorem*{observation*}{Observation}
\newtheorem{exa}[thm]{Example}
\newtheorem*{exa*}{Example}
\newtheorem{df}[thm]{Definition}
\newtheorem*{df*}{Definition}

\newtheorem*{conj*}{Conjecture}

\renewcommand{\geq}{\geqslant}
\renewcommand{\leq}{\leqslant}

\begin{document}
\begin{center}
{\large  \bf The $\mathfrak{a}$-Filter grade of an ideal $\mathfrak{b}$ and $(\mathfrak{a},\mathfrak{b})$-$\mathrm{f}$-modules}

\vspace{0.3cm} Jingwen Shen, Xiaoyan Yang  \\
Department of Mathematics, Northwest Normal University, Lanzhou 730070,
China\\
E-mails: shenjw0609@163.com, yangxy@nwnu.edu.cn
\end{center}
\bigskip
\centerline { \bf  Abstract} Let $\mathfrak{a},\mathfrak{b}$ be two ideals of a commutative noetherian ring $R$ and $M$ a finitely generated $R$-module.~We continue to study $\textrm{f}\textrm{-}\mathrm{grad}_R(\mathfrak{a},\mathfrak{b},M)$ which was introduced in [Bull. Malays. Math. Sci. Soc. 38 (2015) 467--482], some computations and bounds of $\textrm{f}\textrm{-}\mathrm{grad}_R(\mathfrak{a},\mathfrak{b},M)$ are provided.~We also give the structure of $(\mathfrak{a},\mathfrak{b})$-$\mathrm{f}$-modules,~various properties which are analogous to those of Cohen Macaulay modules are discovered.
\leftskip10truemm \rightskip10truemm \noindent \\
\vbox to 0.3cm{}\\
{\it Key Words:} $\mathfrak{a}$-filter regular sequence; $\textrm{f}\textrm{-}\mathrm{grad}_R(\mathfrak{a},\mathfrak{b},M)$; $(\mathfrak{a},\mathfrak{b})$-$\mathrm{f}$-module\\
{\it 2020 Mathematics Subject Classification:} 13C15; 13H10

\leftskip0truemm \rightskip0truemm
\bigskip

\section{\bf Introduction and Preliminaries}

Depth is one of the most fundamental invariant of a noetherian local ring or a finitely generated module.~It is defined in terms of regular sequences, can be measured by the non-vanishing of $\mathrm{Ext}$ modules.~This quantity reflects the application of homological methods to commutative algebra.~Equally important is characterizations of depth in terms of the local cohomology and the Koszul homology.~These two invariants estimate the relationship between depth and dimension of rings or modules and hold an intermediate position between arithmetic and homological algebra.

Filter regular sequence,~which has been introduced in \cite{NS},~is a generalization of regular sequence.~It plays an important role in the theory of local cohomology and has appeared many papers in this area,~for instance \cite{CT,FTZ,GC,KS,LT}.  It is an interesting problem to determine if a given local cohomology module is artinian.~Let $(R,\mathfrak{m})$ be a local ring,~$\mathfrak{a}$ an ideal of $R$ and $M$ a finitely generated $R$-module.~L$\mathrm{\ddot{u}}$ and Tang \cite{LT} considered this problem and defined $\textrm{f}\textrm{-}\mathrm{depth}_R(\mathfrak{a},M)$ as the length of any maximal $M$-filter regular sequence in $\mathfrak{a}$ and proved that the following are equal to $\textrm{f}\textrm{-}\mathrm{depth}_R(\mathfrak{a},M)$,
\begin{center}$\begin{aligned}
&\bullet \mathrm{inf} \{r\in\mathbb{N}_{0}\hspace{0.03cm}|\hspace{0.03cm}\mathrm{dim}_{R}\mathrm{Ext}^{r}_{R}(R/\mathfrak{a},M)>0\}\\
&\bullet \mathrm{inf} \{r\in\mathbb{N}_{0}\hspace{0.03cm}|\hspace{0.03cm}\mathrm{H}^{r}_{\mathfrak{a}}(M)~\mathrm{is~not~Artinian}\}\\
&\bullet n-\mathrm{sup} \{r\in\mathbb{N}_{0}\hspace{0.03cm}|\hspace{0.03cm}\mathrm{dim}_{R}(\mathrm{H}_{i}(y_{1},\cdots,y_{n};M))>0\}
\end{aligned}$\end{center}
where $\mathrm{H}^{\ast}_{a}(M)$ is the $\mathfrak{a}$-local cohomology of $M$, $(y_{1},\cdots,y_{n})=\mathfrak{a}$ and $\mathrm{H}_{i}(y_{1},\cdots,y_{n};M)$ is the $i$-th homology module of the Koszul complex on $y_{1},\cdots,y_{n}$.

Let $\mathfrak{a},\mathfrak{b}$ be two ideals of $R$.~Fathi et al. \cite{FTZ} denoted the supremum of all numbers $n\in \mathbb{N}_{0}$ for which there exists an $\mathfrak{a}$-filter regular $M$-sequence of length $n$ in $\mathfrak{b}$ by $\textrm{f}\textrm{-}\mathrm{grad}_R(\mathfrak{a},\mathfrak{b},M)$.~If $(R,\mathfrak{m})$ is local,~then $\mathfrak{m}$-filter grade of $\mathfrak{b}$ on $M$ is exactly the f-depth of $\mathfrak{b}$ on $M$.~They generalized some results in \cite{LT} and characterized
$\textrm{f}\textrm{-}\mathrm{grad}_R(\mathfrak{a},\mathfrak{b},M)$ to non-local cases.

The natural question arises:~Do the three formulas yield the same invariant $\textrm{f}\textrm{-}\mathrm{grad}_R(\mathfrak{a},\mathfrak{b},M)$? In Section 2,~we give the affirmative answer for this question.~Most importantly,~we provide some computes of $\textrm{f}\textrm{-}\mathrm{grad}_R(\mathfrak{a},\mathfrak{b},M)$ through the depth of the localizations of $M$ at prime ideals of $R$,~and give some inequalities on $\textrm{f}\textrm{-}\mathrm{grad}_R(\mathfrak{a},\mathfrak{b},M)$.

Section 3 is devoted to define an $(\mathfrak{a},\mathfrak{b})$-$\mathrm{f}$-module using $\textrm{f}\textrm{-}\mathrm{grad}_R(\mathfrak{a},\mathfrak{b},M)$ which is similar to Cohen Macaulay modules.~We give relationships between $(\mathfrak{a},\mathfrak{b})$-$\mathrm{f}$-modules and Cohen Macaulay modules and some properties of $(\mathfrak{a},\mathfrak{b})$-$\mathrm{f}$-modules.

We assume throughout this paper that $R$ is a commutative
noetherian ring which is not necessarily local,~$\mathfrak{a}$ and $\mathfrak{b}$ are ideals of $R$,~all modules are finitely generated $R$-modules.

Next we recall some notions which we will need later.

{\bf Associated prime and support.} We write $\mathrm{Spec}R$ for the set of prime ideals of $R$.~For an ideal $\mathfrak{c}$ of $R$ and $\mathfrak{q}\in \mathrm{Spec}R$,~we set
\begin{center}
$\mathcal{V}(\mathfrak{c})=\{\mathfrak{p}\in \mathrm{Spec}R\hspace{0.03cm}|\hspace{0.03cm}\mathfrak{c}\subseteq \mathfrak{p}\}$,
$\mathcal{U}(\mathfrak{q})=\{\mathfrak{u}\in \mathrm{Spec}R\hspace{0.03cm}|\hspace{0.03cm}\mathfrak{u}\subseteq \mathfrak{q}\}$.
\end{center}
Let $M$ be an $R$-module.~The associated prime ideals of $M$,~denoted by $\mathrm{Ass}_{R}M$,~is the set
\begin{center}$\begin{aligned}
\mathrm{Ass}_{R}M
&= \{\mathfrak{p}\in \mathrm{Spec}R\hspace{0.03cm}|\hspace{0.03cm}\mathrm{there~exists}~x\in M~\mathrm{scuh~that}~ \mathrm{ann}x=\mathfrak{p}\}.
\end{aligned}$\end{center}
Fix $\mathfrak{p}\in \mathrm{Spec}R$,~let $M_{\mathfrak{p}}$ denote the localization of $M$ at $\mathfrak{p}$.~The ``large" support of $M$ is
\begin{center}$\begin{aligned}
\mathrm{Supp}_{R}M
&= \{\mathfrak{p}\in \mathrm{Spec}R\hspace{0.03cm}|\hspace{0.03cm}M_{\mathfrak{p}}\neq 0\}.
\end{aligned}$\end{center}
It is well known that $\mathrm{Ass}_{R}M\subseteq \mathrm{Supp}_{R}M$.

The (Krull) dimension of an $R$-module $M$ is
\begin{center}
$\mathrm{dim}_{R}M=\mathrm{sup}\{\mathrm{dim}_{R}R/\mathfrak{p}\hspace{0.03cm}|\hspace{0.03cm}\mathfrak{p}\in \mathrm{Supp}_{R}M\}$.
\end{center}

{\bf Filter regular sequence.} ~Let~$M$ be an $R$-module.~We say that a sequence $x_{1},\cdots,x_{n}~\mathrm{in}~R$ is an $\mathfrak{a}$-filter regular $M$-sequence,~if
$$\xymatrix{\mathrm{Supp}_R(\frac{(x_{1},\cdots,x_{i-1})M:_M x_{i}}{(x_{1},\cdots,x_{i-1})M}) \subseteq \mathcal{V}(\mathfrak{a})}$$
for all $i=1,\cdots,n$.~If $x_{1},\cdots,x_{n}$ belong to $\mathfrak{b}$,~then we say that $x_{1},\cdots,x_{n}$ is an $\mathfrak{a}$-filter regular $M$-sequence in $\mathfrak{b}$.~Note that $x_{1},\cdots,x_{n}$ is an $R$-filter regular $M$-sequence if and only if it is a weak $M$-regular sequence.

The following proposition is included in \cite[Proposition 2.1]{FTZ}.

\begin{prop}\label{prop:2.1}
Let $x_{1},\cdots,x_{n}$ be a sequence in $R$, $M$ an $R$-module and $n\in \mathbb{N}$.~The following are equivalent:

$\mathrm{(1)}$ $x_{1},\cdots,x_{n}$ is an $\mathfrak{a}$-filter regular $M$-sequence;

$\mathrm{(2)}$ $x_{i}\not\in \mathfrak{p}$ for all $\mathfrak{p}\in \mathrm{Ass}_{R}(M/(x_{1},\cdots,x_{i-1})M) \backslash \mathcal{V}(\mathfrak{a})$ and all $i=1,\cdots,n$;

$\mathrm{(3)}$ $x_{1}/1,\cdots,x_{n}/1$ is a weak $M_{\mathfrak{p}}$-regular sequence for all $\mathfrak{p}\in \mathrm{Supp}_{R}M\backslash  \mathcal{V}(\mathfrak{a})$.
\end{prop}

\bigskip
\section{\bf $\mathfrak{a}$-filter grade of an ideal $\mathfrak{b}$ on modules }

This section we consider whether $\textrm{f}\textrm{-}\mathrm{grad}_R(\mathfrak{a},\mathfrak{b},M)$ can be measured by Koszul complexes, The~computation and various inequalities on $\textrm{f}\textrm{-}\mathrm{grad}_R(\mathfrak{a},\mathfrak{b},M)$ are given.

We begin with the following definition in \cite{FTZ}.

\begin{df}\label{df:2.4}
\rm Let $M$ be an $R$-module. Suppose that $\mathrm{Supp}_{R}(M/\mathfrak{b}M)\not\subseteq \mathcal{V}(\mathfrak{a})$.~The $\mathfrak{a}$-filter grade of $\mathfrak{b}$ on $M$ is defined as the length of any maximal $\mathfrak{a}$-filter regular $M$-sequences in $\mathfrak{b}$,~denoted by $\textrm{f}\textrm{-}\mathrm{grad}_R(\mathfrak{a},\mathfrak{b},M)$.~Set $\textrm{f}\textrm{-}\mathrm{grad}_R(\mathfrak{a},\mathfrak{b},M)=\infty$ when $\mathrm{Supp}_{R}(M/\mathfrak{b}M)\subseteq \mathcal{V}(\mathfrak{a})$.
\end{df}

\begin{rem}\label{rem:2.5} \rm
$\mathrm{(1)}$ If $\mathfrak{a}=R$,~then $\textrm{f}\textrm{-}\mathrm{grad}_R(\mathfrak{a},\mathfrak{b},M)=\mathrm{depth}_{R}(\mathfrak{b},M)$ in \cite{IL}.

$\mathrm{(2)}$ If $(R,\mathfrak{m})$ is a local ring,~then $\textrm{f}\textrm{-}\mathrm{grad}_R(R,\mathfrak{m},M)=\mathrm{depth}_{R}M$ in \cite{BH}; $\textrm{f}\textrm{-}\mathrm{grad}_R(\mathfrak{m},\mathfrak{b},M)=\textrm{f}\textrm{-}\mathrm{depth}_{R}(\mathfrak{b},M)$ in \cite{LT}.
\end{rem}

\begin{prop}\label{prop:2.6} Let $y_{1},\cdots,y_{n}\in \mathfrak{b}$ be such that $\mathfrak{b}=(y_{1},\cdots,y_{n})$.~Then
$$\mathrm{f}\textrm{-}\mathrm{grad}_R(\mathfrak{a},\mathfrak{b},M)=n-\mathrm{sup} \{i\hspace{0.03cm}|\hspace{0.03cm}\mathrm{Supp}_R\mathrm{H}_{i}(y_{1},\cdots,y_{n};M)\not\subseteq \mathcal{V}(\mathfrak{a})\}.$$
\end{prop}

\begin{proof}
If $\mathrm{Supp}_R(M/\mathfrak{b}M)\subseteq \mathcal{V}(\mathfrak{a})$,~then $\textrm{f}\textrm{-}\mathrm{grad}_R(\mathfrak{a},\mathfrak{b},M)=\infty$.~We need to show that
\begin{center}
$\mathrm{Supp}_R\mathrm{H}_{i}(y_{1},\cdots,y_{n};M)\subseteq \mathcal{V}(\mathfrak{a})$.
\end{center}
As $\textrm{f}\textrm{-}\mathrm{grad}_R(\mathfrak{a},\mathfrak{b},M)=\infty$, $\mathrm{Supp}_R\mathrm{Ext}^{i}_{R}(R/\mathfrak{b},M)\subseteq \mathcal{V}(\mathfrak{a})$ for all $i> 0$.~So for any $\mathfrak{p}\not\in \mathcal{V}(\mathfrak{a})$ and $i> 0$,~$\mathrm{Ext}^{i}_{R}(R/\mathfrak{b},M)_{\mathfrak{p}}=0$,~i.e.
$\mathrm{inf} \{i\in\mathbb{N}_{0}\hspace{0.03cm}|\hspace{0.03cm}\mathrm{Ext}^{i}_{R}(R/\mathfrak{b},M)_{\mathfrak{p}}\neq 0\}=\infty,$
which implies that
$$\mathrm{sup}\{i\in\mathbb{N}_{0}\hspace{0.03cm}|\hspace{0.03cm}\mathrm{H}_{i}(y_{1},\cdots,y_{n};M)_{\mathfrak{p}}\neq0\}=-\infty.$$
Thus $\mathrm{H}_{i}(y_{1},\cdots,y_{n};M)_{\mathfrak{p}}=0$,~so the result is proved in this case.

Now assume that $\mathrm{Supp}_R(M/\mathfrak{b}M)\not\subseteq \mathcal{V}(\mathfrak{a})$.~Set $r=\textrm{f}\textrm{-}\mathrm{grad}_R(\mathfrak{a},\mathfrak{b},M)$.~We use induction on $r$.~If $r=0$,~then $\mathfrak{b}\subseteq \underset{\mathfrak{p}\in \mathrm{Ass}_{R}M\backslash \mathcal{V}(\mathfrak{a})}{\bigcup}\mathfrak{p}$.~Hence $\mathfrak{b}\subseteq \mathfrak{p}$ for some $\mathfrak{p}\in \mathrm{Ass}_{R}M\backslash \mathcal{V}(\mathfrak{a})$.~Thus there exists $m\in M$ such that $\mathfrak{p}=\mathrm{ann}(m)$.~As $\mathfrak{b}m=0$,~we see that $m\in 0:_{M}\mathfrak{b}\subseteq\mathrm{H}_{n}(y_{1},\cdots,y_{n};M)$,~and then $\mathfrak{p}\in\mathrm{Supp}_R\mathrm{H}_{n}(y_{1},\cdots,y_{n};M)$.~While $\mathfrak{p}\not\in \mathcal{V}(\mathfrak{a})$,~so $$\mathrm{Supp}_R\mathrm{H}_{n}(y_{1},\cdots,y_{n};M)\nsubseteq \mathcal{V}(\mathfrak{a})$$
and the equality holds.~Suppose that $r > 0$.~Let $x\in \mathfrak{b}$ be an $\mathfrak{a}$-filter regular element and $M_{1}=M/xM$.~Since $\textrm{f}\textrm{-}\mathrm{grad}_R(\mathfrak{a},\mathfrak{b},M_{1})=r-1$, one has
$$\mathrm{sup}\{i\hspace{0.03cm}|\hspace{0.03cm}\mathrm{Supp}_R\mathrm{H}_{i}(y_{1},\cdots,y_{n};M_{1})\not\subseteq \mathcal{V}(\mathfrak{a})\}=n-r+1$$by the induction hypothesis.
Note that
\begin{center}$\mathrm{Supp}_R\mathrm{H}_{i}(y_{1},\cdots,y_{n};M)
\subseteq \mathrm{Supp}_RM\cap\mathcal{V}(\mathfrak{b})
= \mathrm{Supp}_R(M/\mathfrak{b}M)$,\end{center}
we have $\mathrm{H}_{i}(y_{1},\cdots,y_{n};M_{1})_{\mathfrak{p}}=0$ for all $i > n-r+1$ and any $\mathfrak{p}\in \mathrm{Supp}_R(M/\mathfrak{b}M)\backslash \mathcal{V}(\mathfrak{a})$, and there exists $\mathfrak{p}\in \mathrm{Supp}_R(M/\mathfrak{b}M)\backslash \mathcal{V}(\mathfrak{a})$ such that $\mathrm{H}_{n-r+1}(y_{1},\cdots,y_{n};M_{1})_{\mathfrak{p}}\neq0$.~For any $\mathfrak{p}\in \mathrm{Supp}_R(M/\mathfrak{b}M)\backslash \mathcal{V}(\mathfrak{a})$,~as $x\in \mathfrak{b}\subseteq \mathfrak{p}$,~we see that $x/1$ is $M_{\mathfrak{p}}$-regular element.~From the short exact sequence
$$\xymatrix@C=20pt@R=0pt{0\ar[r] & M_{\mathfrak{p}}\ar[r]^{x/1} & M_{\mathfrak{p}}\ar[r] & (M_{1})_{\mathfrak{p}}\ar[r] & 0,}$$
we have a long exact sequence
$$\xymatrix@C=15pt@R=5pt{
\cdots\ar[r] & \text{H}_{i}(y_{1}/1,\cdots,y_{n}/1;M_{\mathfrak{p}})\ar[r]^{x/1} & \mathrm{H}_{i}(y_{1}/1,\cdots,y_{n}/1;M_{\mathfrak{p}})\ar[r] & \mathrm{H}_{i}(y_{1}/1,\cdots,y_{n}/1;(M_{1})_{\mathfrak{p}}) \\
\ar[r] & \mathrm{H}_{i-1}(y_{1}/1,\cdots,y_{n}/1;M_{\mathfrak{p}})\ar[r]^{x/1} &\cdots.\hspace{3cm}&}$$
As $\mathrm{H}_{i}(y_{1}/1,\cdots,y_{n}/1;M_{\mathfrak{p}})$ is annihilated by $x/1$,~the above long exact sequence is split into short exact sequences
$$\xymatrix@C=10pt@R=10pt{0\ar[r] & \mathrm{H}_{i}(y_{1}/1,\cdots,y_{n}/1;M_{\mathfrak{p}})\ar[r] & \mathrm{H}_{i}(y_{1}/1,\cdots,y_{n}/1;(M_{1})_{\mathfrak{p}})\ar[r] & \mathrm{H}_{i-1}(y_{1}/1,\cdots,y_{n}/1;M_{\mathfrak{p}})\ar[r] & 0,}$$
that is to say,
$$\xymatrix@C=10pt@R=10pt{0\ar[r] & \mathrm{H}_{i}(y_{1},\cdots,y_{n};M)_{\mathfrak{p}}\ar[r] & \mathrm{H}_{i}(y_{1},\cdots,y_{n};M_{1})_{\mathfrak{p}}\ar[r] & \mathrm{H}_{i-1}(y_{1},\cdots,y_{n};M)_{\mathfrak{p}}\ar[r] & 0.}$$
Then $\mathrm{H}_{i}(y_{1},\cdots,y_{n};M_{1})_{\mathfrak{p}}=0$ for any $i>n-r$ and any $\mathfrak{p}\in \mathrm{Supp}_R(M/\mathfrak{b}M)\backslash \mathcal{V}(\mathfrak{a})$,~and $\mathrm{H}_{n-r}(y_{1},\cdots,y_{n};M_{1})_{\mathfrak{p}}\neq 0$ for some $\mathfrak{p}\in \mathrm{Supp}_R(M/\mathfrak{b}M)\backslash \mathcal{V}(\mathfrak{a})$.~Hence
$$\mathrm{sup} \{i\hspace{0.03cm}|\hspace{0.03cm}\mathrm{Supp}_R\mathrm{H}_{i}(y_{1},\cdots,y_{n};M)\not\subseteq \mathcal{V}(\mathfrak{a})\}=n-r.$$
Hence we obtain the desired equality.
\end{proof}

\begin{rem}\label{rem:2.5} \rm Let $M$ be an $R$-module.
\cite[Theorem 2.2]{FTZ} showed that
\begin{center}$\begin{aligned}
\textrm{f}\textrm{-}\mathrm{grad}_R(\mathfrak{a},\mathfrak{b},M)
&= \mathrm{inf} \{r\in\mathbb{N}_{0}\hspace{0.03cm}|\hspace{0.03cm}\mathrm{Supp}_{R}\mathrm{Ext}^{r}_{R}(R/\mathfrak{b},M)\not\subseteq \mathcal{V}(\mathfrak{a})\} \\
&= \mathrm{inf} \{r\in\mathbb{N}_{0}\hspace{0.03cm}|\hspace{0.03cm}\mathrm{Supp}_{R}\mathrm{H}^{r}_\mathfrak{b}(M)\not\subseteq \mathcal{V}(\mathfrak{a})\}.
\end{aligned}$\end{center}
Hence the three approaches in the introduction yield the same invariant,~that is $\mathfrak{a}$-filter grade of $\mathfrak{b}$ on $M$.
\end{rem}

The theorem below describes the local nature of $\textrm{f}\textrm{-}\mathrm{grad}_R(\mathfrak{a},\mathfrak{b},M)$.

\begin{thm}\label{thm:2.7} For an $R$-module $M$ such that $\mathrm{Supp}_R(M/\mathfrak{b}M)\not\subseteq \mathcal{V}(\mathfrak{a})$, one has that
\begin{center}$\begin{aligned}
\mathrm{f}\textrm{-}\mathrm{grad}_R(\mathfrak{a},\mathfrak{b},M)
&= \mathrm{inf} \{\mathrm{depth}_{R_{\mathfrak{p}}}(\mathfrak{b}R_{\mathfrak{p}},M_{\mathfrak{p}})\hspace{0.03cm}|\hspace{0.03cm}\mathfrak{p}\in \mathrm{Supp}_{R}(M/\mathfrak{b}M)\backslash \mathcal{V}(\mathfrak{a})\}   \\
&= \mathrm{inf} \{\mathrm{depth}_{R_{\mathfrak{p}}}(M_{\mathfrak{p}})\hspace{0.03cm}|\hspace{0.03cm}\mathfrak{p}\in \mathrm{Supp}_{R}(M/\mathfrak{b}M)\backslash \mathcal{V}(\mathfrak{a})\}.
\end{aligned}$\end{center}
\end{thm}

\begin{proof}
Set $t=\mathrm{f}\textrm{-}\mathrm{grad}_R(\mathfrak{a},\mathfrak{b},M)$.~Then $t=\mathrm{inf} \{r\in\mathbb{N}_{0}\hspace{0.03cm}|\hspace{0.03cm}\mathrm{Supp}_{R}\mathrm{Ext}^{r}_{R}(R/\mathfrak{b},M)\not\subseteq \mathcal{V}(\mathfrak{a})\}$.~Note that $\mathrm{Supp}_R\mathrm{Ext}^{i}_{R}(R/\mathfrak{b},M)\subseteq \mathcal{V}(\mathfrak{a})$ for all $i<t$,~so $\mathrm{Ext}^{i}_{R}(R/\mathfrak{b},M)_{\mathfrak{p}}=\mathrm{Ext}^{i}_{R_{\mathfrak{p}}}(R_{\mathfrak{p}}/\mathfrak{b}R_{\mathfrak{p}},M_{\mathfrak{p}})=0$ for any $i<t$ and any $\mathfrak{p}\in \mathrm{Supp}(M/\mathfrak{b}M)\backslash \mathcal{V}(\mathfrak{a})$,~which implies that $\mathrm{depth}_{R_{\mathfrak{p}}}(\mathfrak{b}R_{\mathfrak{p}},M_{\mathfrak{p}})\geq t$.~While $\mathrm{Supp}_R\mathrm{Ext}^{t}_{R}(R/\mathfrak{b},M)\not\subseteq \mathcal{V}(\mathfrak{a})$,~there exists $\mathfrak{p}\in \mathrm{Supp}_R(M/\mathfrak{b}M)\backslash\mathcal{V}(\mathfrak{a})$ such that $\mathrm{Ext}^{t}_{R_{\mathfrak{p}}}(R_{\mathfrak{p}}/\mathfrak{b}R_{\mathfrak{p}},M_{\mathfrak{p}})\neq 0$.~It follows that
$$t=\mathrm{inf} \{\mathrm{depth}_{R_{\mathfrak{p}}}(\mathfrak{b}R_{\mathfrak{p}},M_{\mathfrak{p}})\hspace{0.03cm}|\hspace{0.03cm}\mathfrak{p}\in \mathrm{Supp}_{R}(M/\mathfrak{b}M)\backslash \mathcal{V}(\mathfrak{a})\}.$$
The second statement follows from \cite[Proposition 1.2.10]{BH}.
\end{proof}

\begin{cor}\label{cor:2.8} Let $\mathfrak{a},\mathfrak{a}',\mathfrak{b},\mathfrak{b}'$ be ideals of $R$ and $M$ an $R$-module.

$\mathrm{(1)}$ If $\mathfrak{a}\subseteq\mathfrak{a}'$, then $\mathrm{f}\textrm{-}\mathrm{grad}_R(\mathfrak{a}',\mathfrak{b},M)\leq
\mathrm{f}\textrm{-}\mathrm{grad}_R(\mathfrak{a},\mathfrak{b},M)$.

$\mathrm{(2)}$ If $\mathfrak{b}\subseteq\mathfrak{b}'$, then $\mathrm{f}\textrm{-}\mathrm{grad}_R(\mathfrak{a},\mathfrak{b},M)\leq
\mathrm{f}\textrm{-}\mathrm{grad}_R(\mathfrak{a},\mathfrak{b}',M)$.

$\mathrm{(3)}$ If $\sqrt{a}=\sqrt{a'}$, then $\mathrm{f}\textrm{-}\mathrm{grad}_R(\mathfrak{a},\mathfrak{b},M)=
\mathrm{f}\textrm{-}\mathrm{grad}_R(\mathfrak{a}',\mathfrak{b},M)$.

$\mathrm{(4)}$ If $\sqrt{b}=\sqrt{b'}$, then $\mathrm{f}\textrm{-}\mathrm{grad}_R(\mathfrak{a},\mathfrak{b},M)=
\mathrm{f}\textrm{-}\mathrm{grad}_R(\mathfrak{a},\mathfrak{b}',M)$.
\end{cor}

\begin{proof}
(1) As $\mathfrak{a}\subseteq\mathfrak{a}'$,~we get $\mathcal{V}(\mathfrak{a}')\subseteq \mathcal{V}(\mathfrak{a})$.~We divide three cases to prove it.

Case 1.~$\mathrm{Supp}_R(M/\mathfrak{b}M)\subseteq \mathcal{V}(\mathfrak{a}')$,~then $\mathrm{f}\textrm{-}\mathrm{grad}_R(\mathfrak{a}',\mathfrak{b},M)=\infty$,
~$\mathrm{f}\textrm{-}\mathrm{grad}_R(\mathfrak{a},\mathfrak{b},M)=\infty$.~

Case 2.~$\mathrm{Supp}_R(M/\mathfrak{b}M)\not\subseteq \mathcal{V}(\mathfrak{a}')$ and $\mathrm{Supp}_R(M/\mathfrak{b}M)\subseteq \mathcal{V}(\mathfrak{a})$,~then
$\mathrm{f}\textrm{-}\mathrm{grad}_R(\mathfrak{a},\mathfrak{b},M)=\infty$.~Hence $\mathrm{f}\textrm{-}\mathrm{grad}_R(\mathfrak{a}',\mathfrak{b},M)\leq
\mathrm{f}\textrm{-}\mathrm{grad}_R(\mathfrak{a},\mathfrak{b},M)$.~

Case 3.~$\mathrm{Supp}_R(M/\mathfrak{b}M)\not\subseteq \mathcal{V}(\mathfrak{a})$,~it follows from Theorem \ref{thm:2.7} and $\mathrm{Supp}_R(M/\mathfrak{b}M)\backslash\mathcal{V}(\mathfrak{a})\subseteq
\mathrm{Supp}_R(M/\mathfrak{b}M)\backslash\mathcal{V}(\mathfrak{a}')$.

(2) Note that $\mathrm{Supp}_R(M/\mathfrak{b}'M)\subseteq \mathrm{Supp}_R(M/\mathfrak{b}M)$.~We also divide three cases to prove it.

Case 1.~$\mathrm{Supp}_R(M/\mathfrak{b}M)\subseteq \mathcal{V}(\mathfrak{a})$,~then
$\mathrm{f}\textrm{-}\mathrm{grad}_R(\mathfrak{a},\mathfrak{b}',M)=\infty$,
$\mathrm{f}\textrm{-}\mathrm{grad}_R(\mathfrak{a},\mathfrak{b},M)=\infty$.~

Case 2.~$\mathrm{Supp}_R(M/\mathfrak{b}'M)\subseteq \mathcal{V}(\mathfrak{a})$ and $\mathrm{Supp}_R(M/\mathfrak{b}M)\not\subseteq \mathcal{V}(\mathfrak{a})$,~then
$\mathrm{f}\textrm{-}\mathrm{grad}_R(\mathfrak{a},\mathfrak{b}',M)=\infty$.~Hence
$\mathrm{f}\textrm{-}\mathrm{grad}_R(\mathfrak{a},\mathfrak{b},M)\leq
\mathrm{f}\textrm{-}\mathrm{grad}_R(\mathfrak{a},\mathfrak{b}',M)$.

Case 3.~$\mathrm{Supp}_R(M/\mathfrak{b}'M)\not\subseteq \mathcal{V}(\mathfrak{a})$,~it follows from Theorem \ref{thm:2.7} and $\mathrm{Supp}_R(M/\mathfrak{b}'M)\subseteq\mathrm{Supp}_R(M/\mathfrak{b}M)$.

(3) The equality holds since $\mathcal{V}(\mathfrak{a})= \mathcal{V}(\mathfrak{a'})$.

(4) The equality holds since $\mathrm{Supp}_R(M/\mathfrak{b}M)=\mathrm{Supp}_R(M/\mathfrak{b}'M)$.
\end{proof}

\begin{cor}\label{cor:2.10}
Let $M$ be an $R$-module.~If $\mathrm{Supp}_R(M/\mathfrak{b}M)\nsubseteq \mathcal{V}(\mathfrak{a})$,~then there are inequalities
 \begin{center}
 $\mathrm{depth}_R(\mathfrak{b},M)\leq
\mathrm{f}\textrm{-}\mathrm{depth}_R(\mathfrak{b},M)\leq
\mathrm{f}\textrm{-}\mathrm{grad}_R(\mathfrak{a},\mathfrak{b},M)\leq\mathrm{ht}_M\mathfrak{b}$,
\end{center} where $\mathrm{ht}_M\mathfrak{b}$ is the infimum of lengths of strictly decreasing chains of prime ideals in $\mathrm{Supp}_RM$ starting from a prime ideal containing $\mathfrak{b}$.
In particular,~$\mathrm{f}\textrm{-}\mathrm{grad}_R(\mathfrak{a},\mathfrak{p},M)\leq \mathrm{dim}_{R_{\mathfrak{p}}}M_{\mathfrak{p}}$~for any $\mathfrak{p}\in \mathrm{Supp}_RM\backslash \mathcal{V}(\mathfrak{a})$.
\end{cor}

\begin{proof}
The first inequality follows from \cite[Corollary 1.1.3]{BH},~the second one
is by Corollary \ref{cor:2.8}(1).~It remains to show the third inequality.~For any $\mathfrak{p}\in\mathrm{Supp}_R(M/\mathfrak{b}M)\backslash\mathcal{V}(\mathfrak{a})$,~one has that $\mathrm{depth}_{R_\mathfrak{p}}M_\mathfrak{p}\leq\mathrm{dim}_{R_\mathfrak{p}}M_\mathfrak{p}
=\mathrm{ht}_M\mathfrak{p}$.

~Moreover,~for any $\mathfrak{p}\in \mathrm{Supp}_RM\backslash \mathcal{V}(\mathfrak{a})$, $\mathfrak{p}\in \mathrm{Supp}_R(M/\mathfrak{p}M)\backslash \mathcal{V}(\mathfrak{a})$ since $\mathrm{Supp}_R(M/\mathfrak{p}M)=\mathrm{Supp}_RM \cap \mathcal{V}(\mathfrak{p})$.~Hence $\mathrm{f}\textrm{-}\mathrm{grad}_R(\mathfrak{a},\mathfrak{p},M)\leq \mathrm{dim}_{R_{\mathfrak{p}}}M_{\mathfrak{p}}$.
\end{proof}

The next example shows that the inequalities in the above corollaries can be strict.

\begin{exa}\label{exa:2.9} {\rm Let $K[x_{1},x_{2}]$ be a polynomial ring and $K$ a field.~Set $R=K[x_{1},x_{2},x_{3}]_{(x_{1},x_{2},x_{3})}$ and $M=R\oplus R/(x^{2}_{2},x^{3}_{3})$,~$\mathfrak{a}=(x_{1},x_{2},x_{3})$,~$\mathfrak{b}=(x_{1})$ and $\mathfrak{b}'=(x_{1},x_{2})$.~Then~$0:_{M}x_{1}=0$ and $0:_{M/x_{1}M}x_{2}\subseteq R/(x_{1},x_{2}^{2},x_{3}^{3})\subseteq \mathcal{V}(\mathfrak{a})$.~It follows that $x_{1},x_{2}$ is an $\mathfrak{a}$-filter regular $M$-sequence.~Hence $\mathrm{f}\textrm{-}\mathrm{grad}_R(\mathfrak{a},\mathfrak{b},M)=1<2=\mathrm{f}\textrm{-}\mathrm{grad}_R(\mathfrak{a},\mathfrak{b'},M)$}.
\end{exa}

\begin{cor}\label{cor:2.11} For an $R$-module $M$, one has
\begin{center}
 $\mathrm{f}\textrm{-}\mathrm{grad}_R(\mathfrak{a},\mathfrak{b},M)=
\mathrm{inf}\{\mathrm{f}\textrm{-}\mathrm{grad}_R(\mathfrak{a},\mathfrak{p},M)\hspace{0.03cm}|\hspace{0.03cm}
\mathfrak{p}\in \mathcal{V}(\mathfrak{b})\}$.
\end{center}
\end{cor}

\begin{proof}
If $\mathrm{Supp}_R(M/\mathfrak{b}M)\subseteq \mathcal{V}(\mathfrak{a})$,~then $\mathrm{f}\textrm{-}\mathrm{grad}_R(\mathfrak{a},\mathfrak{b},M)=\infty$.~For every $\mathfrak{p}\in \mathcal{V}(\mathfrak{b})$,~$\mathrm{f}\textrm{-}\mathrm{grad}_R(\mathfrak{a},\\\mathfrak{b},M)\leq
\mathrm{f}\textrm{-}\mathrm{grad}_R(\mathfrak{a},\mathfrak{p},M)$ by Corollary \ref{cor:2.8}.~Hence $\mathrm{inf}\{\mathrm{f}\textrm{-}\mathrm{grad}_R(\mathfrak{a},\mathfrak{p},M)\hspace{0.03cm}|\hspace{0.03cm}
\mathfrak{p}\in \mathcal{V}(\mathfrak{b})\}=\infty$,~the equality holds.~Now assume that $\mathrm{Supp}_R(M/\mathfrak{b}M)\nsubseteq \mathcal{V}(\mathfrak{a})$.~Set $\ell=\mathrm{f}\textrm{-}\mathrm{grad}_R(\mathfrak{a},\mathfrak{b},M)$.~It follows from Corollary \ref{cor:2.8} that $\ell\leq \mathrm{inf}\{\mathrm{f}\textrm{-}\mathrm{grad}_R(\mathfrak{a},\mathfrak{p},M)\hspace{0.03cm}|\hspace{0.03cm}
\mathfrak{p}\in \mathcal{V}(\mathfrak{b})\}$.~On the other hand,~there is $\mathfrak{p}\in \mathcal{V}(\mathfrak{b})$ such that
$\mathrm{depth}_{R_\mathfrak{p}}M_\mathfrak{p}=\ell$ ~by Theorem \ref{thm:2.7}.~But $\mathfrak{p}\not\in\mathcal{V}(\mathfrak{a})$,~so $\mathrm{f}\textrm{-}\mathrm{grad}_R(\mathfrak{a},\mathfrak{p},M)\leq
\mathrm{depth}_{R_\mathfrak{p}}M_\mathfrak{p}=\ell$ by Theorem \ref{thm:2.7} again.~Thus $\ell \geq \mathrm{inf}\{\mathrm{f}\textrm{-}\mathrm{grad}_R(\mathfrak{a},\mathfrak{p},M)\hspace{0.03cm}|\hspace{0.03cm}
\mathfrak{p}\in \mathcal{V}(\mathfrak{b})\}$.~Consequently, $\ell = \mathrm{inf}\{\mathrm{f}\textrm{-}\mathrm{grad}_R(\mathfrak{a},\mathfrak{p},M)\hspace{0.03cm}|\hspace{0.03cm}
\mathfrak{p}\in \mathcal{V}(\mathfrak{b})\}$,~as desired.
\end{proof}

The following result gives a bound of $\mathrm{f}\textrm{-}\mathrm{grad}_R(\mathfrak{a},\mathfrak{b},M)$.

\begin{prop}\label{prop:2.13} Let~$ 0\neq M$ be an $R$-module such that $\mathrm{Supp}_R(M/\mathfrak{b}M)\nsubseteq \mathcal{V}(\mathfrak{a})$. Then for any $\mathfrak{p}\in\mathrm{Supp}_R(M/\mathfrak{b}M)\backslash \mathcal{V}(\mathfrak{a})$,~there exists an inequality
\begin{center}
$\mathrm{f}\textrm{-}\mathrm{grad}_R(\mathfrak{a},\mathfrak{b},M)\leq \mathrm{dim}_RM-\mathrm{dim}_{R}R/\mathfrak{p}$.
\end{center}
\end{prop}

\begin{proof} We have the following inequalities
\begin{center}$\begin{aligned}\mathrm{f}\textrm{-}\mathrm{grad}_R(\mathfrak{a},\mathfrak{b},M)
&\leq\mathrm{depth}_{R_\mathfrak{p}}M_\mathfrak{p}\\
&\leq\mathrm{dim}_{R_\mathfrak{p}}M_\mathfrak{p}\\
&\leq \mathrm{dim}_RM-\mathrm{dim}_{R}R/\mathfrak{p},\end{aligned}$\end{center}
where the first one is by Theorem \ref{thm:2.7}, the second one is by \cite[Theorem 9.3]{IL} since $M_\mathfrak{p}\not\simeq0$ and the third one is by \cite[p.413]{BH},~as claimed.
\end{proof}

The next corollary is a direct consequence of
 the preceding proposition,~which recovers and generalizes known results about the usual ($\mathrm{f}$-)depth.

\begin{cor}\label{cor:2.14} Let ~$ 0\neq M$ be an $R$-module.

$\mathrm{(1)}$ $\mathrm{depth}_{R}(\mathfrak{b},M)\leq \mathrm{dim}_RM-\mathrm{dim}_{R}M/\mathfrak{b}M$.

$\mathrm{(2)}$ If $(R,\mathfrak{m})$ is local,~then $\mathrm{f}\textrm{-}\mathrm{depth}_R(\mathfrak{b},M)\leq \mathrm{dim}_RM-\mathrm{dim}_{R}R/\mathfrak{p}$ for all $\mathfrak{p}\in \mathrm{Supp}_{R}M\backslash \{\mathfrak{m}\}$.

$\mathrm{(3)}$ If $(R,\mathfrak{m})$ is local,~then $\mathrm{depth}_{R}M\leq \mathrm{dim}_{R}M$.
\end{cor}

\bigskip
\section{\bf  $(\mathfrak{a},\mathfrak{b})$-$\mathrm{f}$-modules}
Let~$0\neq M$ be an $R$-module such that $\mathrm{Supp}_R(M/\mathfrak{b}M)\nsubseteq \mathcal{V}(\mathfrak{a})$.~According to Proposition \ref{prop:2.13},~we have an inequality
$$\mathrm{f}\textrm{-}\mathrm{grad}_R(\mathfrak{a},\mathfrak{p},M)\leq \mathrm{dim}_RM-\mathrm{dim}_{R}R/\mathfrak{p}\ \textrm{for\ any}\ \mathfrak{p}\in\mathrm{Supp}_R(M/\mathfrak{b}M)\backslash \mathcal{V}(\mathfrak{a}).$$
This section calls $M$ an $(\mathfrak{a},\mathfrak{b})\text{-}\mathrm{f}$-module whenever the equality holds, some properties that are analogous to those of Cohen Macaulay modules are given.

\begin{df}\label{df:3.1}
\rm Let~$0\neq M$ be an $R$-module such that  $\mathrm{Supp}_R(M/\mathfrak{b}M)\nsubseteq \mathcal{V}(\mathfrak{a})$.~$M$ is said to be $(\mathfrak{a},\mathfrak{b})$-$\mathrm{f}$-module~if the following equality holds
$$\mathrm{f}\textrm{-}\mathrm{grad}_R(\mathfrak{a},\mathfrak{p},M)=\mathrm{dim}_RM-\mathrm{dim}_{R}R/\mathfrak{p}$$
for any $\mathfrak{p}\in\mathrm{Supp}_R(M/\mathfrak{b}M)\backslash \mathcal{V}(\mathfrak{a})$.
\end{df}

Suppose that $\mathrm{Supp}_R(M/\mathfrak{b}M)\nsubseteq \mathcal{V}(\mathfrak{a})$.~If $x$ is $\mathfrak{a}$-filter regular in $\mathfrak{b}$,~then $x\not\in \underset{\mathfrak{p}\in \mathrm{Ass}_{R}M\backslash \mathcal{V}(\mathfrak{a})}{\bigcup}\mathfrak{p}$. Hence $x\not\in \underset{\mathfrak{p}\in \mathrm{min}(\mathrm{Ass}_{R}M)\backslash \mathcal{V}(\mathfrak{a})}{\bigcup}\mathfrak{p}$,~where $\mathrm{min}(\mathrm{Ass}_{R}M)$ is the set of minimal elements of $\mathrm{Ass}_{R}M$. Thus $\mathrm{dim}_R(M/xM)=\mathrm{dim}_RM-1$, which shows that every $\mathfrak{a}$-filter regular sequence in $\mathfrak{b}$ is a part of system of parameters for $M$.

Recall that an $R$-module $M$ is said to be a Cohen Macaulay module~if $\mathrm{depth}_{R_{\mathfrak{m}}}M_{\mathfrak{m}}=\mathrm{dim}_{R_{\mathfrak{m}}}M_{\mathfrak{m}}$ for all maximal ideals $\mathfrak{m}$ of $R$ (details see \cite{BH}).~$M$ is called a $\mathfrak{b}$-Cohen Macaulay module~if the equality $\mathrm{depth}_{R}(\mathfrak{b},M)+\mathrm{dim}_{R}(M/\mathfrak{b}M)=\mathrm{dim}_{R}M$ holds (details see \cite{MA}).

The following proposition gives some examples of $(\mathfrak{a},\mathfrak{b})$-$\mathrm{f}$-modules.

\begin{prop}\label{prop:3.2}

$\mathrm{(1)}$ If $M$ is a Cohen Macaulay module,~then $M$ is an $(\mathfrak{a},\mathfrak{b})$-$\mathrm{f}$-module.

$\mathrm{(2)}$ If $(R,\mathfrak{m})$ is local, $\mathfrak{a},\mathfrak{b}$ proper ideals of $R$~and $M$ a $\mathrm{f}$-module,~then $M$ is an $(\mathfrak{a},\mathfrak{b})$-$\mathrm{f}$-module.

$\mathrm{(3)}$ If $M$ is an $(\mathfrak{a},\mathfrak{b})$-$\mathrm{f}$-module,~then
$\mathrm{f}\textrm{-}\mathrm{grad}_R(\mathfrak{a},\mathfrak{b},M)=\mathrm{dim}_RM-\mathrm{dim}_{R}(M/\mathfrak{b}M)$.~In particular,
$\mathrm{dim}_{R}(M/\mathfrak{b}M)=\mathrm{dim}_{R}R/\mathfrak{p}$
for all $\mathfrak{p}\in \mathrm{Supp}_R(M/\mathfrak{b}M)\backslash \mathcal{V}(\mathfrak{a})$.

$\mathrm{(4)}$ $(R,\mathfrak{b})$-$\mathrm{f}$-modules are exactly $\mathfrak{b}$-Cohen Macaulay modules.
\end{prop}

\begin{proof}
$\mathrm{(1)}$ Since $M$ is a Cohen Macaulay $R$-module, $M_{\mathfrak{p}}$ is a Cohen Macaulay $R_{\mathfrak{p}}$-module for all $\mathfrak{p}\in \mathrm{Supp}_R(M/\mathfrak{b}M)\backslash \mathcal{V}(\mathfrak{a})$.~Set $\mathrm{depth}_{R_{\mathfrak{p}}}M_{\mathfrak{p}}=\mathrm{dim}_{R_{\mathfrak{p}}}M_{\mathfrak{p}}=d$, and let $x_{1}/1,\cdots,x_{d}/1$ be a system of parameters for $M_{\mathfrak{p}}$. Then it is a maximal $M_{\mathfrak{p}}$-regular sequence.~Thus $x_{1},\cdots,x_{d}$ is a maximal $\mathfrak{a}$-filter regular $M$-sequence in $\mathfrak{p}$ by Proposition \ref{prop:2.1}.~So $\mathrm{f}\textrm{-}\mathrm{grad}_R(\mathfrak{a},\mathfrak{p},M)=d=
\mathrm{dim}_{R_{\mathfrak{p}}}M_{\mathfrak{p}}=\mathrm{dim}_{R}M-\mathrm{dim}_{R}R/\mathfrak{p}$, as desired.

$\mathrm{(2)}$ Since $M$ is a f-module,~$\mathrm{f}\textrm{-}\mathrm{depth}_{R}(\mathfrak{p},M)=\mathrm{dim}_{R}M-\mathrm{dim}_{R}R/\mathfrak{p}$ for all $\mathfrak{p}\in \mathrm{Supp}_R(M/\mathfrak{b}M)\backslash \\\mathcal{V}(\mathfrak{m})$.~While $\mathrm{f}\textrm{-}\mathrm{grad}_R(\mathfrak{a},\mathfrak{p},M)\geq \mathrm{f}\textrm{-}\mathrm{depth}_{R}(\mathfrak{p},M)$ by Corollary \ref{cor:2.10} and $\mathrm{f}\textrm{-}\mathrm{grad}_R(\mathfrak{a},\mathfrak{p},M)\leq \mathrm{dim}_{R}M-\mathrm{dim}_{R}R/\mathfrak{p}$ for all $\mathfrak{p}\in \mathrm{Supp}_R(M/\mathfrak{b}M)\backslash \mathcal{V}(\mathfrak{a})$.~Hence $M$ is $(\mathfrak{a},\mathfrak{b})$-$\mathrm{f}$-module.

$\mathrm{(3)}$ Suppose that $x_{1},\cdots,x_{r}$ is a maximal $\mathfrak{a}$-filter regular sequence in $\mathfrak{b}$.~Then it is a part of system of parameters for $M$,~and hence
\begin{center}$\begin{aligned}
\mathrm{dim}_{R}(M/\mathfrak{b}M)
&= \mathrm{dim}_{R}(M/(x_{1},\cdots,x_{r})M)\\
&= \mathrm{dim}_{R}M-r\\
&= \mathrm{dim}_{R}M-\mathrm{f}\textrm{-}\mathrm{grad}_R(\mathfrak{a},\mathfrak{p},M)
\end{aligned}$\end{center}
for all $\mathfrak{p}\in\mathrm{Supp}_R(M/\mathfrak{b}M)\backslash \mathcal{V}(\mathfrak{a})$.~By Corollary \ref{cor:2.11}
$$\mathrm{f}\textrm{-}\mathrm{grad}_R(\mathfrak{a},\mathfrak{b},M)=
\mathrm{inf}\{\mathrm{f}\textrm{-}\mathrm{grad}_R(\mathfrak{a},\mathfrak{p},M)\hspace{0.03cm}|\hspace{0.03cm}
\mathfrak{p}\in \mathcal{V}(\mathfrak{b})\},$$
so $\mathrm{f}\textrm{-}\mathrm{grad}_R(\mathfrak{a},\mathfrak{b},M)=\mathrm{dim}_{R}M-\mathrm{dim}_{R}(M/\mathfrak{b}M)$.~Note that for any $\mathfrak{p}\in\mathrm{Supp}_R(M/\mathfrak{b}M)\backslash \mathcal{V}(\mathfrak{a})$,
$$\mathrm{f}\textrm{-}\mathrm{grad}_R(\mathfrak{a},\mathfrak{p},M)=\mathrm{dim}_{R}M-\mathrm{dim}_{R}R/\mathfrak{p},$$
it follows that $\mathrm{dim}_{R}(M/\mathfrak{b}M)=\mathrm{dim}_{R}R/\mathfrak{p}$.

$\mathrm{(4)}$ This follows from \cite{MA}.
\end{proof}

\begin{prop}\label{prop:3.3} If $M$ is an $(\mathfrak{a},\mathfrak{b})$-$\mathrm{f}$-module,~then $M_{\mathfrak{p}}$ is a Cohen Macaulay module for all $\mathfrak{p}\in \mathrm{Supp}_R(M/\mathfrak{b}M)\backslash \mathcal{V}(\mathfrak{a})$.
\end{prop}

\begin{proof}
Set $\mathrm{f}\textrm{-}\mathrm{grad}_R(\mathfrak{a},\mathfrak{b},M)=r$.~Since $M$ is an $(\mathfrak{a},\mathfrak{b})$-$\mathrm{f}$-module,~we have $\mathrm{ht}_{M}\mathfrak{p}+\mathrm{dim}_{R}R/\mathfrak{p}=\mathrm{dim}_{R}M$ for all $\mathfrak{p}\in \mathrm{Supp}_R(M/\mathfrak{b}M)\backslash \mathcal{V}(\mathfrak{a})$.
So
\begin{center}$\begin{aligned}
\mathrm{ht}_{M}\mathfrak{p}
&= \mathrm{dim}_{R}M-\mathrm{dim}_{R}R/\mathfrak{p}\\
&= \mathrm{f}\textrm{-}\mathrm{grad}_R(\mathfrak{a},\mathfrak{p},M)=r.\end{aligned}$\end{center}
Thus $\mathrm{dim}_{R_{\mathfrak{p}}}M_{\mathfrak{p}}=r$.~Let $x_{1},\cdots,x_{r}$ be a maximal $\mathfrak{a}$-filter regular $M$-sequence in $\mathfrak{b}$.~Then it is a part of system of parameters for $M$.~So $x_{1}/1,\cdots,x_{r}/1$ is a system of parameters for $M_{\mathfrak{p}}$.~While $x_{1}/1,\cdots,x_{r}/1$ is a $M_{\mathfrak{p}}$-regular sequence by Proposition \ref{prop:2.1},~thus $M_{\mathfrak{p}}$ is a Cohen Macaulay module, as claimed.
\end{proof}

\begin{prop}\label{prop:3.3} Let $\mathfrak{a},\mathfrak{a}',\mathfrak{b}$ be ideals of $R$ with $\mathfrak{a}'\subseteq\mathfrak{a}$ and $M$ an $R$-module.~If $M$ is an $(\mathfrak{a},\mathfrak{b})$-$\mathrm{f}$-module,~then $M$ is an $(\mathfrak{a}',\mathfrak{b})$-$\mathrm{f}$-module.
\end{prop}

\begin{proof}
Since $M$ is an $(\mathfrak{a},\mathfrak{b})$-$\mathrm{f}$-module, it follows that
$$\mathrm{f}\textrm{-}\mathrm{grad}_R(\mathfrak{a},\mathfrak{p},M)=\mathrm{dim}_RM-\mathrm{dim}_{R}R/\mathfrak{p}$$
for any $\mathfrak{p}\in\mathrm{Supp}_R(M/\mathfrak{b}M)\backslash \mathcal{V}(\mathfrak{a})$.
As $\mathfrak{a}'\subseteq\mathfrak{a}$,
$\mathrm{f}\textrm{-}\mathrm{grad}_R(\mathfrak{a}',\mathfrak{p},M)
\geq \mathrm{f}\textrm{-}\mathrm{grad}_R(\mathfrak{a},\mathfrak{b},M)
=\mathrm{dim}_RM-\mathrm{dim}_{R}R/\mathfrak{p}$ by Corollary \ref{cor:2.8}.~Note that
$$\mathrm{f}\textrm{-}\mathrm{grad}_R(\mathfrak{a}',\mathfrak{p},M)\leq \mathrm{dim}_RM-\mathrm{dim}_{R}R/\mathfrak{p}$$
for all $\mathfrak{p}\in\mathrm{Supp}_R(M/\mathfrak{b}M)\backslash \mathcal{V}(\mathfrak{a})$ by Proposition \ref{prop:2.13}.~Hence,~$M$ is an $(\mathfrak{a}',\mathfrak{b})$-$\mathrm{f}$-module.
\end{proof}

\begin{lem}\label{lem:3.4} Let~$M$ be an $R$-module and $x\in \mathfrak{b}$ be $\mathfrak{a}$-filter regular element.~Then $M$ is an $(\mathfrak{a},\mathfrak{b})$-$\mathrm{f}$-module if and only if $M/xM$ is an $(\mathfrak{a},\mathfrak{b}/(x))$-$\mathrm{f}$-module over $R/xR$.
\end{lem}

\begin{proof}
Note that for all $\mathfrak{p}\in\mathrm{Supp}_R(M/\mathfrak{b}M)\backslash \mathcal{V}(\mathfrak{a})$,~there is an isomorphism $R'/\mathfrak{p}'\cong R/\mathfrak{p}$,~where $R'=R/(x)$ and $\mathfrak{p}'=\mathfrak{p}/(x)$.
~So~the proof follows from Definition \ref{df:3.1} and the next equality
\begin{center}
$\textrm{f}\textrm{-}\mathrm{grad}_{R/(x)}(\mathfrak{a},\mathfrak{p}/(x),M/xM)=\mathrm{f}\textrm{-}\mathrm{grad}_R(\mathfrak{a},\mathfrak{p},M)-1$.
\end{center}
This completes the proof.
\end{proof}

\begin{cor}\label{cor:3.5}
Let~$M$ be an $R$-module and $x_1,\cdots,x_r$ an $\mathfrak{a}$-filter regular sequence in $\mathfrak{b}$. Then~$M$ is an $(\mathfrak{a},\mathfrak{b})$-$\mathrm{f}$-module if and only if $M/(x_1,\cdots,x_r)M$ is an $(\mathfrak{a},\mathfrak{b}/(x_1,\cdots,x_r))$-$\mathrm{f}$-module over $R/(x_1,\cdots,x_r)R$.
\end{cor}

An ideal $\mathfrak{a}$ of $R$ is said to be primary if $ab\in \mathfrak{a}$ and $a\not\in \mathfrak{a}$ implies that $b^n\in \mathfrak{a}$ for some $n\geq1$.~If $\mathfrak{a}$ is a primary ideal and $\mathfrak{p}=\sqrt{\mathfrak{a}}$ the radical of $\mathfrak{a}$,~then $\mathfrak{a}$ is said to be $\mathfrak{p}$-primary.

\begin{prop}\label{prop:3.5} Let~$M$ be an $(\mathfrak{a},\mathfrak{b})$-$\mathrm{f}$-module such that $\mathfrak{p}\in \mathrm{Supp}_{R}M\backslash \mathcal{V}(\mathfrak{a})$.~If $\mathfrak{b}$ is $\mathfrak{p}$-primary,~then
\begin{center}
$\mathrm{dim}_RM=\mathrm{dim}_RR/\mathfrak{q}$ for some $\mathfrak{q}\in (\mathrm{Ass}_RM\backslash \mathcal{V}(\mathfrak{a}))\cap \mathcal{U}(\mathfrak{p})$.
\end{center}
\end{prop}

\begin{proof}
If $\mathrm{f}\textrm{-}\mathrm{grad}_R(\mathfrak{a},\mathfrak{b},M)=0$,~then $\mathrm{f}\textrm{-}\mathrm{grad}_R(\mathfrak{a},\mathfrak{p},M)=0$.~While $\mathfrak{p}\in \mathrm{Supp}_{R}M\backslash \mathcal{V}(\mathfrak{a})$,~so $\mathfrak{q}\subseteq \mathfrak{p}$ for some $\mathfrak{q}\in \mathrm{Ass}_RM\backslash \mathcal{V}(\mathfrak{a})$.~Note
that $\mathfrak{p}=\sqrt{\mathfrak{b}}$,~hence $\mathfrak{p}\in \mathrm{Supp}_R(M/\mathfrak{b}M)\backslash \mathcal{V}(\mathfrak{a})$ and $\mathfrak{p}$ is the smallest prime ideal containing $\mathfrak{b}$ by \cite[Proposition 4.1]{AM},~it follows that $\mathrm{dim}_RM= \mathrm{dim}_R(M/\mathfrak{b}M)=\mathrm{dim}_RR/\mathfrak{p}$ and $\mathfrak{p}=\mathfrak{q}$.~Now suppose $\mathrm{f}\textrm{-}\mathrm{grad}_R(\mathfrak{a},\mathfrak{b},M)>0$.~Then there is an $\mathfrak{a}$-filter regular element $x\in \mathfrak{b}$.~Set $\bar{M}=M/xM$.~Then $\bar{M}$ is an $(\mathfrak{a},\mathfrak{b}/(x))$-$\mathrm{f}$-module by Lemma \ref{lem:3.4}.~By induction,~$\mathrm{dim}_{R}\bar{M}=\mathrm{dim}_{R}R/\mathfrak{q}_{1}$ for some $\mathfrak{q}_{1}\in (\mathrm{Ass}_{R}\bar{M}\backslash \mathcal{V}(\mathfrak{a})) \cap \mathcal{U}(\mathfrak{p})$.~Then $\mathfrak{q}_{1}\in \mathrm{Supp}_{R}\bar{M}\backslash \mathcal{V}(\mathfrak{a})$.~In particular,~$\mathfrak{q}_{1}\in \mathrm{Supp}_{R}M\backslash \mathcal{V}(\mathfrak{a})$ and $x\in \mathfrak{q}_{1}$.~Then there exists $\mathfrak{q}\in\mathrm{Ass}_{R}M\backslash \mathcal{V}(\mathfrak{a})$ such that $\mathfrak{q}\subsetneq \mathfrak{q}_{1}\subseteq \mathfrak{p}$.~Hence $\mathrm{dim}_{R}R/\mathfrak{q}\geq 1+\mathrm{dim}_{R}R/\mathfrak{q}_{1}=1+\mathrm{dim}_{R}\bar{M}=\mathrm{dim}_{R}M$.~This proves that $\mathrm{dim}_{R}M=\mathrm{dim}_{R}R/\mathfrak{q}$,~where $\mathfrak{q}\in (\mathrm{Ass}_{R}M \backslash \mathcal{V}(\mathfrak{a}))\cap \mathcal{U}(\mathfrak{p})$.
\end{proof}

\bigskip \centerline {\bf ACKNOWLEDGEMENTS} This research was partially supported by National Natural Science Foundation of China (11761060,11901463), Improvement of Young Teachers$'$ Scientific Research Ability (NWNU-LKQN-18-30) and Innovation Ability Enhancement Project of Gansu Higher Education Institutions (2019A-002).

\end{document}